\documentclass[12pt, reqno]{amsart}
\usepackage{amsmath, amsthm, amscd, amsfonts, amssymb, graphicx, color,tikz}
\usepackage[bookmarksnumbered, colorlinks, plainpages]{hyperref}
\usepackage{mathrsfs}
\usepackage{enumerate}

\newtheorem{theorem}{Theorem}[section]
\newtheorem{lemma}[theorem]{Lemma}
\newtheorem{proposition}[theorem]{Proposition}
\newtheorem{corollary}[theorem]{Corollary}
\theoremstyle{definition}
\newtheorem{definition}[theorem]{Definition}
\newtheorem{example}[theorem]{Example}

\newtheorem{question}[theorem]{Question}
\numberwithin{equation}{section}

\newcommand{\veps}{\varepsilon}
\newcommand{\CC}{\mathbb C}
\newcommand{\RR}{\mathbb R}

\newcommand{\QQ}{\mathbb Q}
\newcommand{\TT}{\mathbb T}
\newcommand{\NN}{\mathbb N}
\newcommand{\ZZ}{\mathbb Z}

\newcommand{\dwa}{{\mathcal D}_{\mathrm{w.a.}}}
\newcommand{\ldens}{\underline{\mathrm{dens}}}
\newcommand{\ddens}{{\mathcal D}_{\mathrm{dens}}}

\allowdisplaybreaks[4]

\begin{document}
\setcounter{page}{1}

\title[Rearrangement universality]
{Universality of general Dirichlet series with respect to translations and rearrangements}
\date{\today}

\author[F. Bayart]{Frédéric Bayart}

\address{Laboratoire de Math\'ematiques Blaise Pascal UMR 6620 CNRS, Universit\'e Clermont Auvergne, Campus universitaire des C\'ezeaux, 3 place Vasarely, 63178 Aubi\`ere Cedex, France.}
\email{frederic.bayart@uca.fr}




\begin{abstract}
We give sufficient conditions for a general Dirichlet series to be universal with respect to translations or rearrangements.
\end{abstract}
\maketitle

\section{Introduction}
\subsection{Universality with respect to translations}

A general Dirichlet series  is a series $\sum_{n=1}^{+\infty} a_n e^{-\lambda_n s}$ where $(a_n)\subset\CC^\mathbb N$, $s\in\CC$ and $\lambda=(\lambda_n)$ is an increasing sequence of nonnegative real numbers tending to $+\infty$, called a \emph{frequency}. The two most natural examples are $(\lambda_n)=(n)$ which gives rise to power series and $(\lambda_n)=(\log n)$, the case of ordinary Dirichlet series.
To a Dirichlet series $D=\sum_{n=1}^{+\infty} a_n e^{-\lambda_n s}$ we will associate three abscissas, its abscissa of convergence, 
$$\sigma_c(D):=\inf\left\{\Re e(s):\ \sum_n a_n e^{-\lambda_n s}\textrm{ converges}\right\},$$
 its abscissa of absolute convergence 
$$\sigma_a(D):=\inf\left\{\sigma\in\RR:\ \sum_n |a_n|e^{-\lambda_n\sigma}\textrm{ converges}\right\},$$
and also
$$\sigma_2(D):=\inf\left\{\sigma\in\RR:\ \sum_n |a_n|^2e^{-2\lambda_n\sigma}\textrm{ converges}\right\}.$$
It is well-known that $\sum_n a_ne^{-\lambda_n s}$ converges in the half-plane $\CC_{\sigma_c(f)}$, where $\CC_{\sigma}=\{s:\ \Re e(s)>\sigma\}$ and that it defines a holomorphic function there. 

In this paper, we are interested in universal properties of Dirichlet series with respect to vertical translations. The first result in that direction comes from the seminal work of Voronin \cite{VORONIN} and says that the Riemann-Zeta function $\zeta(s)=\sum_{n\geq 1}n^{-s}$ is universal in the critical strip. Voronin's theorem may be stated as follows:
\begin{quote}
Let $K$ be a compact subset of $\{1/2<\Re e(s)<1\}$ with connected complement, let $f$ be a non-vanishing function continuous on $K$ and holomorphic in the interior of $K$. Then 
$$\ldens\left\{\tau\geq 0:\ \sup_{s\in K}|\zeta(s+i\tau)-f(s)|<\veps\right\}>0$$
where $\ldens(A)$ denotes the lower density of $A\subset\RR_+$. 
\end{quote}

Since then, many works have been done on this subject proving universality (and more) in various classes of Dirichlet series, see the survey paper \cite{Matsu15}.
Most of these examples are ordinary Dirichlet series. 
Regarding general Dirichlet series, it is worth mentioning the universality of Lerch Zeta functions defined by 
 $$\zeta(s;\alpha,\lambda)=\sum_{n=0}^{+\infty}e^{2\pi i \lambda n}(n+\alpha)^{-s}$$
 where $0<\alpha\leq 1$ and $\lambda\in\mathbb R$.
 The universality of $\zeta(\cdot;\alpha,\lambda)$ can be expressed in the following form:
\begin{quote}
Let $0<\alpha\leq 1$ be transcendental and let $\lambda\notin\ZZ$. Let $K$ be a compact subset of $\{1/2<\Re e(s)<1\}$ with connected complement, let $f$ be a function continuous on $K$ and holomorphic in the interior of $K$. Then 
$$\ldens\left\{\tau\geq 0:\ \sup_{s\in K}|\zeta(s+i\tau;\alpha,\lambda)-f(s)|<\veps\right\}>0.$$
\end{quote}
Observe now that we do not assume that $f$ does not vanish and we will say that $\zeta(\cdot;\alpha,\lambda)$ is strongly universal in $\{1/2<\Re e(s)<1\}$. 
In fact, one can also prove universal properties of $\zeta(\cdot;\alpha,\lambda)$ when $\alpha$ is rational (but the method is rather different and reduces to Voronin's type results), see \cite{GaLa02} for details. 	

The first attempt to get a general result of (strong) universality for general Dirichlet series was done in \cite{LSS03}.
Assume that $(\lambda_n)$ is linearly independent over $\QQ$. Let $\sigma_0<\sigma_a(D)$, put $r(x)=\sum_{\lambda_n\leq x}1$ and $c_n=a_n\exp(-\lambda_n \sigma_a(D))$. We assume that
\begin{enumerate}[(i)]
\item $D$ cannot be represented as an Euler product;
\item $D$ can be continued meromorphically to $\{\Re e(s)>\sigma_0\}$, and holomorphically
in $\{\sigma_0<\Re e(s)<\sigma_a(D)\}$;
\item For $\sigma>\sigma_0$, $D(s)=O(|t|^\alpha)$ with some $\alpha>0$ (here and elsewhere, we write $s=\sigma+it$);
\item For $\sigma>\sigma_0$, 
$$\int_{-T}^T |D(\sigma+it)|^2 dt =O(T);$$
\item $r(x)=Cx^\kappa+O(1)$ with $\kappa>1$;
\item $|c_n|$ is bounded and $\sum_{\lambda_n\leq x}|c_n|^2=\theta r(x)(1+o(1))$ for some $\theta>0$.
\end{enumerate}
Then $D$ is strongly universal in $\{\sigma_0<\Re e(s)<\sigma_a(D)\}$. 

This result gives sense to the conjecture of Linnik and Ibragimov that all ``reasonable'' Dirichlet series meromorphically continuable to the left of the half-plane of absolute convergence are universal in some suitable region. However, the conditions (v) and (vi) are rather rigid (and they are not satisfied by the Lerch Zeta functions) whereas, given a 
Dirichlet series, it is not so easy to testify whether (iv) holds true.

\medskip


Let us compare the Riemann Zeta function and Lerch Zeta functions, $\lambda\notin\ZZ$. In the former case $\sigma_c(\zeta)=\sigma_a(\zeta)=1$ and $\zeta$ is universal in a strip where it is defined by analytic continuation. In the latter case,  $\sigma_c(\zeta(\cdot;\alpha,\lambda)=0<\sigma_a(\zeta(\cdot;\alpha,\lambda))=1$ and $\xi(\cdot;\alpha,\lambda)$ is universal in a strip where the Dirichlet series converges.  In this paper, we will mainly consider Dirichlet series $D$ for which $\sigma_c(D)<\sigma_a(D)$ 
and we will investigate universal properties of $D$ in the strip $\{\sigma_c(D)<\Re e(s)<\sigma_a(D)\}$, or in a smaller one. In particular, the condition $\sigma_c(D)<\sigma_a(D)$ implies that the sequence 
$(\lambda_n)$ satisfies $\limsup_{n\to+\infty}\frac{\log n}{\lambda_n}>0$, which means that $(\lambda_n)$ cannot grow too fast to $+\infty$.
 
A consequence of our
results is the following theorem showing strong universality by only looking at the frequency and the coefficients of a Dirichlet series. 

\begin{theorem}\label{thm:universaltranslation}
Let $\sigma_0\in \mathbb R$ and let $D(s)=\sum_{n\geq 1}a(n)e^{-\lambda(n)s}$ be a Dirichlet series satisfying the following assumptions:
\begin{enumerate}[(i)]
\item $\sigma_2(D)\leq \sigma_0<\sigma_a(D)$;
\item for all $n\geq 1$, $a(n)=\rho(n)e^{i\omega n}$ for some $\omega\notin 2\pi\mathbb Z$ and $\rho(n)\geq 0$;
\item $\lambda,\rho:[1,+\infty)\to [0,+\infty)$ are $\mathcal C^2$-functions with $\lambda$ increasing and tending to $+\infty$ and $\lambda'$ nonincreasing and tending to $0$;
\item for all $\sigma>\sigma_0$, the function $x\mapsto \rho(x)e^{-\lambda(x)\sigma}$ is nonincreasing;
\item for all $\alpha,\beta>0$, there exist $C>0$ and $x_0\geq 1$ such that, for all $x\geq x_0$, 
$$\sum_{\lambda(n)\in \left[x,x+\frac\alpha{x^2}\right]}|a(n)|\geq Ce^{(\sigma_a(D)-\beta)x};$$
\item the sequence $(\lambda(n))$ is $\mathbb Q$-linearly independent.
\end{enumerate}
Let $K$ be a compact subset of $\{\sigma_0<\Re e(s)<\sigma_a(D)\}$ with connected complement, let $f$ be a function continuous on $K$ and holomorphic in the interior of $K$. Then 
$$\ldens\left\{\tau>0:\ \sup_{s\in K}|D(s+i\tau)-f(s)|<\veps\right\}>0.$$
\end{theorem}

This theorem leads to several very simple examples: the first one generalizes Lerch Zeta functions.
\begin{example}\label{ex:1}
Let $P\in\mathbb R_d[X]$ with $d\geq 1$ and $\lim_{+\infty}P=+\infty$, let $Q\in\mathbb R_{d-1}[X]$, let $\omega\in\mathbb R\backslash 2\pi\ZZ$ and let $\gamma\in\mathbb R$. Assume moreover that $(\log(P(n))_{n\geq 1}$ is $\mathbb Q$-linearly independent. Then the Dirichlet series $D(s)=\sum_{n\geq 1}Q(n)(\log n)^\gamma e^{i\omega n}(P(n))^{-s}$
is strongly universal in $\{(2d-1)/2d<\Re e(s)<1\}$.
\end{example}

One can also find examples where the frequency grows very slowly with a half-plane of universality.

\begin{example}\label{ex:2}
Let $\gamma\geq e-1$ be such that $(\log\log(n+\gamma))_{n\geq 1}$ is $\mathbb Q$-linearly independent and let $\omega\notin2\pi\ZZ$. Let 
$$D(s)=\sum_{n\geq 1}\frac{e^{i\omega n}}n e^{-(\log\log(n+\gamma))s}.$$
Then $D$ is strongly universal in $\{\Re e(s)<1\}$.
\end{example}

Let us comment the assumptions of Theorem \ref{thm:universaltranslation}. (iii) is a way to quantify that $\lambda$ does not go to quickly to $+\infty$. (ii) and (iv) are the natural conditions to ensure that $\sigma_c(D)\leq \sigma_0$ whereas (vi) is mandatory to apply Kronecker's theorem. Only (v) could be surprizing. The Bohr-Cahen formula for $\sigma_a(D)$ ensures that, for all $\beta>0$, for all $N\geq 1$, there exists $x\geq N$ such that
$\sum_{\lambda(n)\leq x}|a(n)|\geq C e^{(\sigma_a(D)-\beta)x}$. Condition (v) means that the sequence $(\lambda(n))$ grows and the sequence $(a(n))$ decreases sufficiently slowly so that this inequality remains true if we restrict the sum to a small neighbourhood of $x$. 

Among the classical Dirichlet series, it seems that the prime number series $D(s)=\sum_{n\geq 1}p_n^{-s}$, where $(p_n)$ is the increasing sequence of prime numbers, has not been investigated from the point of view of universality. Observe that it more difficult to extend meromorphically $D$ to $\mathbb C_{1/2}$ and that, if we do not assume the Riemann hypothesis, we just know that it extends meromorphically to $\mathbb C_{1/2}$ by removing some horizontal line segments (see Section \ref{sec:primeuniversality}). Nevertheless, from the universality of $\log\zeta$, we will be able to deduce that of $D$.

\begin{theorem}\label{thm:primeuniversal}
Let $D(s)=\sum_{n\geq 1}p_n^{-s}$. Let $K$ be a compact subset of $\{1/2<\Re e(s)<1\}$ with connected complement, let $f$ be a function continuous on $K$ and holomorphic in the interior of $K$.  
Then the set of posive real numbers $\tau$ such that $D(\cdot+i\tau)$ is well-defined on $K$ and satisfies 
$$\sup_{s\in K}|D(s+i\tau)-f(s)|<\veps$$
has positive lower density.
\end{theorem}

\subsection{Rearrangement universality}
The second problem we face in this paper is that of rearrangement universality of general Dirichlet series. Let $X$ be a topological vector space and let $(x_n)$ be a sequence of vectors in $X$. Assume that the series $\sum_n x_n$ is conditionally convergent, namely that $\sum_n x_n$ converges but that there exists a permutation $\sigma$ of $\mathbb N$ such that $\sum_n x_{\sigma(n)}$ diverges. A natural question is to study the sum range of $\sum x_n$, namely the set of elements $x\in X$ such that $\sum_{n}x_{\sigma(n)}$ converges to $x$ for some permutation $\sigma$. The extremal behaviour is attained when the sum range is the whole space $X$. This is always the case when $X$ is the real line by a famous theorem of Riemann; when $X$ is a finite-dimensional vector space, Steinitz theorem asserts that the sum range of $\sum x_n$ is always an affine subset of $X$ and it also gives a description of this sum range.

A natural example of conditionally convergent series is given by general Dirichlet series
for which $\sigma_c(D)<\sigma_a(D)$. In that case, the Dirichlet series $\sum_n a_n e^{-\lambda_n s}$ converges conditionally in the strip 
$\left\{s:\ \sigma_c(D)<\Re e(s)<\sigma_a (D)\right\}$. Is this series universal with respect to rearrangements in this strip?

\begin{definition}
Let $D(s)=\sum_n a_n e^{-\lambda_n s}$ be a Dirichlet series with $\sigma_c(D)<\sigma_a(D)$. 
We say that $D$ is rearrangement universal if, for all $f\in H(\Omega)$, where $\Omega$ is the strip $\{s:\ \sigma_c(D)<\Re e(s)<\sigma_a(D)\}$,  there exists a permutation $\sigma$ of $\NN$ such that 
$\sum_n a_{\sigma(n)}e^{-\lambda_{\sigma(n)}s}$ converges to $f$ in $H(\Omega)$. 
\end{definition}

This question seems to have been considered for the first time in \cite{GM21}, where the authors show that for almost all choices of signs $(\theta_n)$, $\sum_n \theta_n n^{-s}$ and $\sum_n \theta_n p_n^{-s}$ are (locally) rearrangement universal in the smaller strip $\{s\in\CC:\ 1/2<\Re e(s)<1\}$.

Our aim, in this paper, is to prove the following sufficient condition for rearrangement universality of Dirichlet series.

\begin{theorem}\label{thm:rearrangement}
Let $D(s)=\sum_n a_n e^{-\lambda_n s}$ be a Dirichlet series with $\sigma_c(D)<\sigma_a(D)$. 
Assume that for all $\alpha,\beta>0$, there exist $C>0$ and $x_0\geq 1$ such that, for all $x\geq x_0$, 
$$\sum_{\lambda(n)\in \left[x,x+\frac\alpha{x^2}\right]}|a(n)|\geq Ce^{(\sigma_a(D)-\beta)x}.$$
Then $D$ is rearrangement universal. 
\end{theorem}

Of course since the condition appearing in Theorem \ref{thm:rearrangement} also appears
in Theorem \ref{thm:universaltranslation}, Examples \ref{ex:1} and \ref{ex:2}
yield rearrangement universal Dirichlet series. There are also some other examples, like
$D(s)=\sum_n (-1)^n n^{-s}$ and $D(s)=\sum_n (-1)^n p_n^{-s}$, or more generally $D(s)=\sum_n z_n n^{-s}$ and $D(s)=\sum_n z_n p_n^{-s}$ for all $(z_n)\in\mathbb T^\infty$ such that $\sigma_c(D)<\sigma_a(D)$, where $\mathbb T^\infty=\{(z_n):\ |z_n|=1\textrm{ for all }n\geq 1\}$.  

It does not come as a surprise that universality with respect to translations and rearrangement universality are linked. For instance, Voronin's original proof was based on 
Pecherski{\v{i}}'s rearrangement theorem in Hilbert spaces.

\subsection{Organisation of the paper}

The paper is organized as follows: Section \ref{SEC:LEMMA} contains some preliminary lemmas whereas Section \ref{SEC:PROOF} is devoted to the proof of Theorem \ref{thm:rearrangement}. The proof of Theorem \ref{thm:universaltranslation}
is divided into Sections \ref{SEC:CVMEASURE} and \ref{SEC:SUPPORT}.
Examples are detailed in Section \ref{SEC:EXAMPLES}.

\section{Useful results}\label{SEC:LEMMA}

We shall use the following lemma on functions with exponential growth. Several variants of it have already appeared in the literature. 

\begin{lemma}\label{lem:exponentialgrowth}
Let $f$ be an entire function of exponential type. Assume that there exists $d>0$ and a sequence $(x_j)$ of real numbers tending to $+\infty$ such that, 
for all $j\geq 1$, $|f(x_j)|\geq e^{-dxj}$.  Then there exist $\alpha>0$ and a sequence of real numbers $(y_j)$ tending to $+\infty$ such that, for all $j\geq 1$, for all $x\in \left[y_j,y_j+\frac{\delta}{y_j^2}\right]$, $|f(x)|\geq e^{-dy_j}/2$.  
\end{lemma}
\begin{proof}
Extracting if necessary, we may and shall assume that $x_j\geq j$ for all $j\geq 1$. Since $f$ is an entire function of exponential type, we may write it
$$f(z)=\sum_{n\geq 0}a_n z^n$$
where the coefficients $a_n$ satisfy $|a_n|\leq CR^n/n!$ for all $n\geq 0$, where $C$ 
and $R$ depend only on $f$,. For $N\geq 1$, we set $P_N(z)=\sum_{n=0}^N a_n z^n$. Let now $j\geq 2$ and $x\in [x_j-1,x_j+1]$. Then
\begin{align*}
|f(x)-P_N(x)|&\leq C\sum_{n=N+1}^{+\infty} |a_n|\cdot |x^n|\\
&\leq C \frac{R^{N+1}(x_j+1)^{N+1}}{(N+1)!}\sum_{n=0}^{+\infty}\frac{R^n (x_j+1)^n}{n!}\\
&\leq C \left(\frac{eR(x_j+1)}{N+1}\right)^{N+1}\exp\big(R(x_j+1)\big)
\end{align*}
by Stirling formula. Let $M\geq 1$ be sufficiently large so that 
$$-M\log(M)+M\log(eR)+R<-d$$
and set $N_j=\lfloor Mx_j\rfloor$. Then
\begin{align*}
|f(x)-P_{N_j}(x)|&\leq C' \exp\big(Mx\log(eR)+Rx-Mx\log(M)\big)\\
&=o(e^{-dx}).
\end{align*}
Now, $|f|$ attains its maximal value on $[x_j-1/j,x_j+1/j]$ at some $y_j$. Let $I_j=[y_j,y_j+\alpha/y_j^2]$ where $\alpha=1/8M^2$. Then, assuming that $x\in I_j$, by Markov's inequality for real polynomials, 
\begin{align*}
|f(x)|&\geq |f(y_j)|-|f(x)-P_{N_j}(x)|-|P_{N_j}(x)-P_{N_j}(y_j)|-|P_{N_j}(y_j)-f(y_j)|\\
&\geq |f(y_j)| -N_j^2 |x-y_j| \sup_{[x_j-1/j,x_j+1/j]}|P_{N_j}|+o(e^{-dy_j})\\
&\geq |f(y_j)|-\frac{1}4\big(|f(y_j)|+o(e^{-dy_j})\big)+o(e^{-dy_j})\\
&\geq \frac 34 |f(y_j)|+o(e^{-dy_j})\\
&\geq \frac 34 e^{-dx_j}+o(e^{-dy_j})\\
&\geq \frac 12 e^{-dy_j}
\end{align*}
provided $j$ is large enough.
\end{proof}

Our next lemma is a fundamental result giving a sufficient condition for a series in a nuclear Fréchet space to be rearrangement universal. It is due to Banaszczyk \cite{Ban90}.

\begin{lemma}\label{lem:rearrangementuniversal}
Let $X$ be a nuclear Fréchet space and let $(x_n)$ be a sequence of elements of $X$ such that 
$\sum_n x_n$ is convergent. Assume moreover that for all $\phi\in X^*$, $\sum_n |\phi(x_n)|=+\infty$. Then for all $x\in X$, there exists a permutation $\sigma$ of $\mathbb N$ such that $x=\sum_n x_{\sigma(n)}$. 
\end{lemma}

We will be able to apply this lemma because for any strip $\Omega\subset\mathbb C$, $H(\Omega)$ is nuclear as the following lemma indicates (see \cite[p. 499]{Jarbook} for a proof).
\begin{lemma}
For every nonempty open set $\Omega\subset\CC$, the Fréchet space $H(\Omega)$ of all holomorphic functions on $\Omega$ is nuclear.
\end{lemma}

To prove the density of vertical translates of Dirichlet series, we will use the following lemma showing the density of some sums of Dirichlet series. It can be found in \cite[Theorem 6.3.10]{Lau96}.

\begin{lemma}\label{lem:densitysums}
Let $U$ be a simply connected domain of $\mathbb C$ and let $(f_n)$ be a sequence in $H(U)$ satisfying:
\begin{enumerate}[a)]
\item if $\mu$ is a complex Borel measure on $(\mathbb C,\mathcal B(\mathbb C))$ with compact support contained in $U$ such that there exists $r\geq 0$ with $\int_U s^r d\mu\neq 0$, then 
$$\sum_n \left|\int_U f_nd\mu\right|=+\infty;$$
\item the series $\sum_{n\geq 1}f_n$ converges in $H(U)$;
\item for any compact set $K\subset U$, the series $\sum_{n\geq 1}\sup_{s\in K}|f_n(s)|^2$ converges. 
\end{enumerate}
Then the set of all convergent series $\sum_{n\geq 1}z_n f_n$ with $z\in\TT^\infty$	 is dense
in $H(U)$.
\end{lemma}

We shall use the following lemma on exponential sums (see \cite[p. 206]{IK04}).
\begin{lemma}\label{lem:vdc}
 Let $a<b$, let $f,g$ be two $\mathcal C^2$-functions and let $\alpha,\beta\in\mathbb R,\ \veps\in(0,1)$. We assume that $f'$ is monotonic with $\alpha\leq f'\leq\beta$. Then 
$$\sum_{n=a}^b g(n)e^{2\pi i f(n)}=\sum_{\alpha-\veps<m<\beta+\veps}\int_a^b g(x)e^{2\pi i (f(x)-mx)}dx+O\left(G\left(\veps^{-1}+\log(\beta-\alpha+2)\right)\right)$$
where the implied constant is absolute and $G=|g(b)|+\int_a^b |g'(y)|dy$.
\end{lemma}

To evaluate some mean squares of Dirichlet polynomials, we will use the classical Montgomery-Vaughan inequality (\cite{MV74}):

\begin{lemma}\label{lem:mv}
 Let $n\geq 1$, let $\lambda_1,\cdots,\lambda_n$ be distinct real numbers and let $u_1,\dots,u_n$ be complex numbers. Then 
 $$\left|\sum_{r\neq s}\frac{\overline{u_r}u_s}{\lambda_r-\lambda_s}\right|\leq \frac{3\pi}2\sum_{r=1}^n |u_r|^2\delta_r^{-1}$$
 where $\delta_r=\min_{s\neq r}|\lambda_r-\lambda_s|$.
\end{lemma}

\section{Proof of Theorem \ref{thm:rearrangement}}\label{SEC:PROOF}

By Lemma \ref{lem:rearrangementuniversal}, we just need to prove that for all $\phi\in H(\Omega)^*$, $\sum_n |\phi(a_n e^{-\lambda_ns})|=+\infty$. By Hahn-Banach and Riesz theorems, there exist a compact set $K\subset \Omega$ and 
$\mu\in\mathcal M(\Omega)$ such that, for all $f\in H(U)$, $\phi(f)=\int_K fd\mu$. Let us set $b=\max(\Re e(s):z\in K)<\sigma_a(D)$ and let us denote
$$\mathcal L_\mu(z)=\int_K e^{-zs}d\mu(s).$$
One has to prove $\sum_n |a_n|\cdot |\mathcal L_\mu(\lambda_n)|=+\infty$. We first observe that $\mathcal L_\mu$ is not identically zero. Indeed, $\mathcal L_\mu^{(n)}(0)=(-1)^n \phi(s^n)$ and the 
polynomials are dense in $H(\Omega)$. It is easy to check that $\mathcal L_\mu$ is an entire function of exponential type. Moreover, Lemma 11.15 of \cite{BM09} says that
$$\limsup_{x\to+\infty}\frac{\log |\mathcal L_\mu(x)|}{x}\geq -b.$$
Let $\beta\in (0,(\sigma_a(D)-b)/2)$. There exists a sequence of real numbers $(x_j)$, tending to $+\infty$, such that, for all $j\geq 1$,
$$|\mathcal L_\mu(x_j)|\geq e^{-(b+\beta)x_j}.$$

Let $\alpha$ and $(y_j)$ be defined by Lemma \ref{lem:exponentialgrowth}. Then, for all $j\geq 1$ sufficiently large 
\begin{align*}
\sum_n  |a_n| \cdot |\mathcal L_\mu(\lambda_n)|&\geq \frac 12 e^{-(b+\beta)y_j}\sum_{\lambda_n \in [y_j,y_j+\alpha/y_j^2]} |a_n|\\
&\geq C_{\alpha,\beta} e^{-(b+\beta)y_j}e^{(\sigma_a(D)-\beta)y_j}\\
&\geq C_{\alpha,\beta} e^{(\sigma_a(D)-b-2\beta)y_j}.
\end{align*}
The right handside of this inequality may be choosen as large as we want, which ends up the proof.

\begin{example}
Let $D(s)=\sum_n (-1)^{n}n^{-s}$. Then $D$ is rearrangement universal.
\end{example}

\begin{proof}
We just write, for any $\alpha,\beta,x>0$, 
\begin{align*}
\sum_{n:\ \log(n)\in \left[x,x+\frac\alpha{x^2}\right]}1&\geq C\left(e^{x+\frac{\alpha}{x^2}}-e^x\right)\\
&\geq C'\frac{e^x}{x^2}\\
&\geq C'' e^{(1-\beta)x}.
\end{align*}
\end{proof}

\begin{example}\label{ex:primerearrangement}
Let $D(s)=\sum_n (-1)^n p_n^{-s}$. Then $D$ is rearrangement universal.
\end{example}

\begin{proof}
Using Hadamard - De la Vallée Poussin estimate,
$$\Pi(x)=\int_2^x \frac{du}{\log u}+O(xe^{-c\sqrt{\log x}}),$$
we write, for any $\alpha,\beta,x>0$, 
\begin{align*}
\sum_{n:\ \log(p_n)\in \left[x,x+\frac\alpha{x^2}\right]}1&= \int_{e^x}^{e^{x+\frac\alpha{x^2}}}\frac{dx}{\log x}+O\left(e^{x+\frac{\alpha}{x^2}}e^{-c\sqrt{\log\left(x+\frac\alpha{x^2}\right)}}\right)\\
&\geq \frac{e^x \left(e^{\frac\alpha{x^2}}-1\right)}{x+\frac\alpha{x^2}}+O\left(e^{x+\frac{\alpha}{x^2}}e^{-c\sqrt{\log\left(x+\frac\alpha{x^2}\right)}}\right)\\
&\geq C\frac{e^x}{x^2}\\
&\geq C' e^{(1-\beta)x}.
\end{align*}
\end{proof}

\section{Convergence of a family of measures}\label{SEC:CVMEASURE}

Let $D=\sum_n a(n)e^{-\lambda(n)s}$ be a Dirichlet series with $\sigma_2(D)<+\infty$
and let $\Omega$ be the half-plane $\Omega=\{\Re e(s)>\sigma_2(D)\}$.
Define for $z$ belonging to the infinite polycircle $\TT^\infty$, $s\in \Omega$ and $N\geq 1$, 
$$D(s,z)=\sum_{n=1}^{+\infty}a(n)z_n e^{-\lambda(n)s},\ D_N(s,z)=\sum_{n=1}^{N}a(n)z_ne^{-\lambda(n)s}.$$
Then for any $s\in\Omega$,  $(D_N(s,\cdot))$ is a martingale in $(\TT^\infty, \mathcal B(\TT^\infty),m_\infty)$ where 
$m_\infty$ is the Haar measure on $\TT^\infty$, since $(z_n)$ is a sequence of independent variables with mean $0$. Moreover, for $s=\sigma+it\in \Omega$,
$$\mathbb E(|D_N(s,\cdot)|^2)\leq \sum_{n=1}^{+\infty}|a(n)|^2 e^{-2\lambda(n)\sigma}<+\infty$$
since $\sigma>\sigma_2(D)$. By Doob's theorem, $(D_N(s,z))_N$ converges for almost all $z\in\TT^\infty$. This implies that, for all $\sigma_1>\sigma_2(D)$, 
the Dirichlet series $D(\cdot,z)$ converges uniformly on all compact subsets of $\{\Re e(s)>\sigma_1\}$ for almost all $z\in\TT^\infty$. Taking a countable intersection, this implies that for almost all $z\in\TT^\infty$, $D(\cdot,z)$ converges 
uniformly on all compact subsets of $\Omega$. Therefore, $D(\cdot,z)$ defines an $H(\Omega)$-valued random element on $(\TT^{\infty},\mathcal B(\TT^\infty),m_\infty)$.
We will denote by $P_{D}$ be the distribution of $D(\cdot,z)$, namely
$$P_{D}(A)=m_\infty\left(\left\{z\in\TT^\infty:\ D(s,z)\in A\right\}\right),\ A\in\mathcal B(H(\Omega)).$$

We intend on the one hand to study the support of $P_D$ and on the other hand to link it with some probability measure associated to translates of $D$.
We need to introduce a definition:
\begin{definition}
Let $\sigma_0\in\mathbb R$. We say that a Dirichlet series $D(s)=\sum_{n}a(n)e^{-\lambda(n)s}$ with finite abscissa of convergence belongs to $\dwa(\sigma_0)$ provided
\begin{enumerate}
\item it extends holomorphically to $\Omega=\{\Re e(s)>\sigma_0\}$;
\item $\sigma_2(D)\leq\sigma_0$;
\item for all $\sigma_1>\sigma_0$, there exists $A,B>0$ such that, for all $s=\sigma+it$
with $\sigma\geq\sigma_1$, $|D(\sigma+it)|\leq A+|t|^B$;
\item for all $\sigma_2>\sigma_1>\sigma_0$, 
$$\sup_{\sigma\in[\sigma_1,\sigma_2]}\sup_{T>0}\frac 1T\int_0^T |D(\sigma+it)|^2dt<+\infty;$$
\item the sequence $(\lambda(n))$ is $\mathbb Q$-linearly independent. 
\end{enumerate}
\end{definition}

Let $D=\sum_{n}a(n)e^{-\lambda(n)s}\in\dwa(\sigma_0)$ and let $\Omega=\{\Re e(s)>\sigma_0\}$. Then for all $\tau\in\mathbb R$, $D(\cdot+i\tau)\in H(\Omega)$
the space of holomorphic functions on $\Omega$. For $T>0$,  we define a probability measure
$\nu_{T,D}$ on $(H(\Omega),\mathcal B(H(\Omega)))$ by
\begin{equation*}\label{eq:nutd}
 \nu_{T,D}(A)=\frac 1T\textrm{meas}\left(\left\{\tau\in[0,T]:\ D(\cdot+i\tau)\in A\right\}\right),\ A\in\mathcal B(H(\Omega)).
\end{equation*}

The following result was proved in \cite{Duy12} (see also \cite{GLS06}).
\begin{theorem}\label{thm:duy}
Let $\sigma_0\in\mathbb R$ and let $D\in\dwa(\sigma_0)$. Then the probability measure
$\nu_{T,D}$ converges weakly to $P_D$ as $T\to+\infty$.
\end{theorem}

Our first task is to exhibit Dirichlet series belonging to $\dwa(\sigma_0)$ by only looking at the behaviour of the sequences of its coefficients and frequencies.

\begin{theorem}\label{thm:dirichletseriesindwa}
Let $\sigma_0\in\mathbb R$, let $D(s)=\sum_{n=1}^{+\infty}a(n)e^{-\lambda(n)s}$ satisfying the following assumptions:
\begin{itemize}
\item[$\diamond$] $\sigma_2(D)\leq \sigma_0$;
\item[$\diamond$] for all $n\geq 1$, $a(n)=\rho(n)e^{i\omega n}$ for some $\omega\notin 2\pi\mathbb Z$ and $\rho(n)\geq 0$;
\item[$\diamond$] $\lambda,\rho:[1,+\infty)\to [0,+\infty)$ are $\mathcal C^2$-functions with $\lambda$ increasing and tending to $+\infty$ and $\lambda'$ nonincreasing and tending to $0$;
\item[$\diamond$] for all $\sigma>\sigma_0$, the function $x\mapsto \rho(x)e^{-\lambda(x)\sigma}$ is nonincreasing;
\item[$\diamond$] the sequence $(\lambda(n))$ is $\mathbb Q$-linearly independent.
\end{itemize}
Then $D$ belongs to $\dwa(\sigma_0)$. 
\end{theorem}
\begin{proof}
Assumptions (1), (2) and (5) of the definition of $\dwa(\sigma_0)$ are immediately satisfied. (3) follows from the fact that a Dirichlet series has finite order in its half-plane of convergence (see \cite[Therem 12]{HarRi} for instance). Let us prove the delicate part, namely (4). Let $\sigma_2>\sigma_1>\sigma_0$, let $\sigma\in[\sigma_1,\sigma_2]$ and let $t\geq 0$. For $1\leq x\leq N$ and $s=\sigma+it$, we write
\begin{equation}\label{eq:dsdwa1}
D(s)=D_{x-1}(s)+\sum_{n=x}^N a(n)e^{-\lambda(n)s}+\sum_{n=N+1}^{+\infty}a(n)e^{-\lambda(n)s}
\end{equation}
where $D_y(s)=\sum_{n\leq y}a(n)e^{-\lambda(n)s}$ denotes the partial sum of $D$. 
We shall first estimate 
$$\sum_{n=x}^N a(n)e^{-\lambda(n)s}=\sum_{n=x}^N \rho(n)e^{-\lambda(n)\sigma}e^{i(\omega n-t \lambda(n))}.$$
One intends to apply Lemma \ref{lem:vdc} to 
\begin{align*}
 g(u)&=\rho(u)e^{-\lambda(u)\sigma}\\
 f(u)&=\frac 1{2\pi}\big(\omega u- t \lambda(u)\big).
\end{align*}
To do this, we choose $x\geq 1$ so that 
\begin{equation}\label{eq:dsdwa2}
t\lambda'(x)\leq \frac 12\textrm{dist}\left(\frac{\omega}{2\pi},\mathbb Z\right).
\end{equation}
Then for all $u\in[x,N]$, we get
$$\frac{\omega}{2\pi}\geq f'(u)=\frac{\omega}{2\pi}-t\lambda'(u)\geq\frac{\omega}{2\pi}-\frac 12\textrm{dist}\left(\frac{\omega}{2\pi},\mathbb Z\right).$$
This leads us to set $\alpha=\frac\omega{2\pi}-\frac 12\textrm{dist}\left(\frac{\omega}{2\pi},\mathbb Z\right)$, $\beta=\frac\omega{2\pi}$ and $\veps=\frac 14\textrm{dist}\left(\frac{\omega}{2\pi},\mathbb Z\right)$, so that the sum appearing in Lemma \ref{lem:vdc} is empty. Moreover, since $g$ is assumed to be nonincreasing and positive,
$$|g(N)|+\int_x^N |g'(y)|dy\leq g(x).$$
Thus, letting $N$ to $+\infty$ in \eqref{eq:dsdwa1}, we find that
$$D(s)=D_{x-1}(s)+R(x,s)$$
with 
$$|R(x,s)|\leq C(\omega)\rho(x)e^{-\lambda(x)\sigma}$$
for those $t$ and $x$ satisfying \eqref{eq:dsdwa2} (in this proof, the notation $C(x_1,\dots,x_p)$ will denote a positive function depending only on the parameters $x_1,\dots,x_p$, 
 whose value may change from line to line). 

Let now $T>0$ be large and let $x\geq 1$ be such that $T\lambda'(x)=\frac 12\textrm{dist}\left(\frac{\omega}{2\pi},\mathbb Z\right)$ so that \eqref{eq:dsdwa2} is satisfied for all $t\in[0,T]$. In particular,
\begin{align*}
\frac 1T\int_0^T |R(x,\sigma+it)|^2 dt&\leq C(\omega)\rho^2(x)e^{-2\lambda(x)\sigma}\\
&\leq C(\omega)\rho^2(1)e^{-2\lambda(1)\sigma}\\
&\leq C(D,\sigma_1).
\end{align*}
Let us turn to $\int_0^T |D_{x-1}(\sigma+it)|^2 dt$. For technical reasons, we write
$$D_{x-1}(s)=a(1)e^{-\lambda(1)s}+\sum_{n=2}^{x-1}a(n)e^{-\lambda(n)s}.$$

We use the Montgomery-Vaughan theorem  to get 
$$\begin{array}{l}
\displaystyle \int_{0}^T \left|\sum_{n=2}^{x-1} a(n)e^{-\lambda(n)s}\right|^2 dt\\
\displaystyle\quad\quad \leq 
T\sum_{n=2}^{x-1}\rho^2(n)e^{-2\lambda(n)\sigma}+
\left|\sum_{2\leq n\neq m\leq x-1}\!\frac{a(n)e^{-\lambda(n)\sigma}\overline{a(m)}e^{-\lambda(m)\sigma}}{\lambda_n-\lambda_m}\left(e^{(\lambda(n)-\lambda(m))iT}-1\right)\right|\\
\displaystyle \quad\quad\leq  T\sum_{n=1}^{+\infty}\rho^2(n)e^{-2\lambda(n)\sigma}+{3\pi}
\sum_{n=2}^{x-1}\frac{\rho^2(n)e^{-2\lambda(n)\sigma}}{|\lambda(n)-\lambda(n-1)|}\\
\displaystyle \quad\quad\leq T\sum_{n=1}^{+\infty}\rho^2(n)e^{-2\lambda(n)\sigma}+{3\pi}
\sum_{n=2}^{x-1}\frac{\rho^2(n)e^{-2\lambda(n)\sigma}}{\lambda'(n)}
\end{array}$$
so that we finally obtain
$$\int_{0}^T |D_{x-1}(\sigma+it)|^2dt\leq 4T \sum_{n=1}^{+\infty}\rho^2(n)e^{-2\lambda(n)\sigma}+6\pi \sum_{n=1}^{x-1}\frac{\rho^2(n)e^{-2\lambda(n)\sigma}}{\lambda'(n)}.$$
We handle the last sum by summing by parts, setting $S(n)=\sum_{k=1}^n \rho^2(k)e^{-2\lambda(k)\sigma}$ so that
\begin{align*}
\sum_{n=1}^{x-1}\frac{\rho^2(n)e^{-2\lambda(n)\sigma}}{\lambda'(n)}
&\leq \sum_{n=1}^{x-1}S(n)\left(\frac{1}{\lambda'(n)}-\frac{1}{\lambda'(n+1)}\right)+S(x-1)\frac{1}{\lambda'(x)}.
\end{align*}
Now, since $\lambda'$ is nonincreasing and since $(S(n))$ is bounded by some constant $C(D,\sigma_1)$, we find
\begin{align*}  
\sum_{n=1}^{x-1}\frac{\rho^2(n)e^{-2\lambda(n)\sigma}}{\lambda'(n)}&\leq C(D,\sigma_1)\sum_{n=1}^{x-1}\left(\frac1{\lambda'(n+1)}-\frac 1{\lambda'(n)}\right)+\frac{C(D,\sigma_1)}{\lambda'(x)}\\
&\leq \frac{C(D,\sigma_1)}{\lambda'(x)}\\
&\leq C(D,\sigma_1) T.
\end{align*}
Finally we have shown as needed that, for all $\sigma\in [\sigma_1,\sigma_2]$, 
$$\frac 1T\int_0^T |D(\sigma+it)|^2 dt\leq C(D,\sigma_1).$$
\end{proof}

\section{Support of a measure and density of translations}\label{SEC:SUPPORT}

The next step in our work is to identify the support of the measure $P_D$ introduced in the previous section. We recall that if $S$ is a separable topological space and $P$ is a probability measure on $(S,\mathcal B(S))$, the support of $P$ is defined by 
$$S_P=\bigcap_{\substack{C\mathrm{ closed} \\ P(C)=1}}C.$$
The support $S_P$ consists of all $x\in S$ such that, for every open neighbourhood $V$ of $x$, $P(V)>0$. Now, let $X$ be an $S$-valued random element defined on a certain probability space. Then the support of the distribution $P(X\in A)$, $A\in\mathcal B(S)$, is called the support of the random variable $X$ and is denoted by $S_X$. 

The following lemma (see \cite[Thm 1.7.10]{Lau96}) will provide us with the support of $P_D$:
\begin{lemma}\label{lem:support}
Let $U$ be a simply connected domain of $\mathbb C$ and let $(X_n)$ be a sequence of independent $H(U)$-valued random variable defined on the same probability space. Assume that the series $\sum_n X_n$ converges almost surely. Then the support of the sum of this series is the closure of the set of all $f\in H(U)$ which may be
written as a convergent series
$$f=\sum_{n=1}^{+\infty}f_n,\ f_n\in S_{X_n}.$$
\end{lemma}

We next introduce a class of Dirichlet series for which the support of $P_D$ will be the whole space $H(U)$. 

\begin{definition}
Let $\sigma_0\in\mathbb R$. We say that a Dirichlet series $D=\sum_n a(n)e^{-\lambda(n)s}$ belongs to $\ddens(\sigma_0)$ provided 
\begin{enumerate}
\item $\sigma_2(D)\leq \sigma_0<\sigma_a(D)$;
\item for all $\alpha,\beta>0$, there exist $C>0$ and $x_0\geq 1$ such that, for all $x\geq x_0$, 
$$\sum_{\lambda(n)\in \left[x,x+\frac\alpha{x^2}\right]}|a(n)|\geq Ce^{(\sigma_a(D)-\beta)x}.$$
\end{enumerate}
\end{definition}

Let $D=\sum_{n}a(n)e^{-\lambda(n)s}\in\ddens(\sigma_0)$ and let $U$ be the strip $\{\sigma_0<\Re e(s)<\sigma_a(D)\}$.
As mentioned above, $D(s,z)=\sum_{n\geq 1}a(n)z_n e^{-\lambda(n)s}$ defines an $H(U)$-valued random element on $(\TT^\infty,\mathcal B(\TT^\infty),m_\infty)$.
We denote by $P_{D,U}$ its distribution which is the restriction of $P_D$ to $H(U)$.

\begin{theorem}\label{thm:support}
Let $\sigma_0\in\RR$, let $D\in\ddens(\sigma_0)$ and let $U$ be the strip $\{\sigma_0<\Re e(s)<\sigma_a(D)\}$. Then the support
of $P_{D,U}$ is $H(U)$.
\end{theorem}

\begin{proof}
Let $w\in\TT^\infty$ be such that the series $D(\cdot,w)$ converges in $\{\Re e(s)>\sigma_0)$. We apply Lemma \ref{lem:support} with $X_n(z)=a(n)z_n w_n e^{-\lambda(n)s}$, $z\in\TT^\infty$. The sequence $(X_n)$ is a sequence of independent random variables on $(\TT^\infty,\mathcal B(\TT^\infty),m_\infty)$. The support of each $X_n$ is the set 
$\left\{a(n)\xi e^{-\lambda(n)s}:\ \xi\in\TT\right\}$. Therefore, the support of $P_{D,U}$ is the closure of all convergent series $\sum_{n=1}^{+\infty}a(n)z_n e^{-\lambda(n)s}$, with $z\in\TT^\infty$. One has to show that this latter set is the whole space $H(U)$. This is done 
by applying Lemma \ref{lem:densitysums} with $f_n=a(n)w_n e^{-\lambda(n)s}$. Indeed, conditions b) and c) are satisfied because $\sigma_0\geq \sigma_2(D)$ and because of the choice of $w$. 
Regarding a), we consider $\mu$ a complex measure with compact support $K$ contained in $U$ and such that there exists $r\geq 0$ with $\int_U s^r d\mu\neq 0$. Let $\mathcal L_\mu$
be its Laplace transform, $\mathcal L_\mu(z)=\int_U e^{-sz}d\mu$. $\mathcal L_\mu$ is nonzero since $\mathcal L_\mu^{(r)}(0)=(-1)^r \int_U s^r d\mu$. Then, arguing as in the proof of Theorem \ref{thm:rearrangement}, 
$$\sum_{n=1}^{+\infty}\left|\int_U f_n d\mu\right|=\sum_n |a(n)|\cdot |\mathcal L_\mu(\lambda(n))|=+\infty$$
so that a) is satisfied.
\end{proof}

Finally, if we combine Theorems \ref{thm:duy} and \ref{thm:support}, we find a sufficient condition for a general Dirichlet series to be strongly universal in some strip.

\begin{corollary}
Let $\sigma_0\in\mathbb R$ and let $D\in\dwa(\sigma_0)\cap\ddens(\sigma_0)$. Let $K$ be a compact subset of the strip $\Omega=\{\sigma_0<\Re e (s)<\sigma_a(D)\}$ with connected complement. Let $f$ be a continuous function on $K$ which is holomorphic inside $K$. Then for all $\veps>0$, 
$$\ldens \left\{\tau\in\mathbb R_+:\ \sup_{s\in K} |D(s+i\tau)-f(s)|<\veps\right\}>0.$$
\end{corollary}

In particular, in view of Theorem \ref{thm:dirichletseriesindwa}, this yields Theorem \ref{thm:universaltranslation}.

\section{Examples}\label{SEC:EXAMPLES}

\subsection{Applications}

One has now to exhibit concrete examples of Dirichlet series in $\dwa(\sigma_0)\cap\ddens(\sigma_0)$. The following very easy lemma will be helpful:
\begin{lemma}\label{lem:reciprocalpolynomial}
Let $P(X)=\sum_{k=0}^d b_k X^k$ be a polynomial of degree $d$, with $b_d>0$. Then there exist $x_0,y_0>0$ such that $P$ induces a bijection from $[x_0,+\infty)$ to $[y_0,+\infty)$, and 
$$P^{-1}(x)=_{+\infty}\frac{1}{b_d^{1/d}}x^{1/d}-\frac{b_{d-1}}{b_d^{(d-1)/d}}+o(1).$$
\end{lemma}
\begin{proof}
Firstly it is easy to show that $P^{-1}(x)\sim_{+\infty} \frac{1}{b_d^{1/d}}x^{1/d}$. Then, 
write $P^{-1}(x)=\frac{1}{b_d^{1/d}}x^{1/d}(1+\veps(x))$ with $\veps(x)=_{+\infty}o(1)$. From $P(P^{-1}(x))=x$ we obtain
$$x+dx\veps(x)+o(x\veps(x))+\frac{b_{d-1}}{b_d^{(d-1)/d}}x^{(d-1)/d}+o(x^{(d-1)/d})=x$$
which in turn yields 
$$\veps(x)=\frac{-b_{d-1}}{b_d^{(d-1)/d}}x^{-1/d}+o(x^{-1/d}).$$
\end{proof}

We then obtain a large class of Dirichlet series belonging to $\dwa\cap\ddens$.
\begin{proposition}
Let $P\in\mathbb R_d[X]$ with $d\geq 1$ and $\lim_{+\infty}P=+\infty$, let $Q\in\mathbb R_{d-1}[X]$, let $\omega\in\mathbb R\backslash 2\pi\ZZ$ and let $\gamma\in\mathbb R$. Assume moreover that $(\log(P(n))_{n\geq 1}$ is $\mathbb Q$-linearly independent. Then the Dirichlet series $D(s)=\sum_{n\geq 1}Q(n)(\log n)^\gamma e^{i\omega n}(P(n))^{-s}$ belongs to $\dwa((2d-1)/2d)\cap\ddens((2d-1)/2d)$.
\end{proposition}
For instance, we may choose $P(n)=(n+\beta)^d$ for some transcendental number $\beta$.
\begin{proof}
Without loss of generality, we may assume that the sequence $(P(n))_{n\geq 1}$ is positive and increasing and that the sequence $(Q(n))_{n\geq 1}$ is positive.
We first observe that $\sigma_a(D)=1$, $\sigma_c(D)=0$ and $\sigma_2(D)=(2d-1)/2d$. We set $\lambda(x)=\log(P(x))$ and $\rho(x)=Q(x)(\log x)^{\gamma}$. We first verify that $D$ belongs to $\ddens$.
Let $\alpha,\beta>0$. We observe that for $x$ large enough
\begin{align*}
\lambda(n)\in\left[x,x+\frac\alpha{x^2}\right]\iff n\in \left[P^{-1}(e^x),P^{-1}\left(e^{x+\frac\alpha{x^2}}\right)\right].
\end{align*}
Using Lemma \ref{lem:reciprocalpolynomial}, we find that 
\begin{align*}
 \mathrm{card}\left\{n:\ \lambda(n)\in\left[x,x+\frac\alpha{x^2}\right]\right\}&
 \geq C\frac{e^{x/d}}{x^{2}}
\end{align*}
(in this proof, the letter $C$ means some constant depending on $P$, $Q$, $\alpha$, $\beta$ and $\gamma$ but not on $x$,
and the value of $C$ may change from line to line).
Therefore 
\begin{align*}
 \sum_{\lambda(n)\in \left[x,x+\frac\alpha{x^2}\right]}|a(n)|&\geq C \frac{e^{x/d}}{x^{2}} e^{x\frac{d-1}d\left(1-\frac\beta 2\right)}\\
 &\geq C e^{(1-\beta )x}.
\end{align*}
That $D$ belongs to $\dwa((2d-1)/2d)$ is easy. Indeed, if $\sigma>\sigma_2(D)$, then $d\sigma>d-1$ and 
$$\rho(x)e^{-\lambda(x)\sigma}=\frac{Q(x)(\log x)^{\gamma}}{[P(x)]^\sigma}$$
is eventually decreasing and convex. The remaining assumptions come from
$\lambda'(x)\sim_{+\infty}d/x$, the definitions of $\lambda$ and $\rho$, and standard calculus.
\end{proof}

We provide a second example with a half-plane of universality. This is done via a frequency which grows very slowly.
\begin{proposition}
Let $\gamma\geq e-1$ be such that $(\log\log(n+\gamma))_{n\geq 1}$ is $\mathbb Q$-linearly independent and let $\omega\notin2\pi\ZZ$. Let 
$$D(s)=\sum_{n\geq 1}\frac{e^{i\omega n}}n e^{-(\log\log(n+\gamma))s}.$$
Then for all $\sigma_0<1$, $D$ belongs to $\dwa(\sigma_0)\cap\ddens(\sigma_0)$.
\end{proposition}
\begin{proof}
It is clear that $\sigma_c(D)=\sigma_2(D)=-\infty$ whereas $\sigma_a(D)=1$. Let $\sigma_0<1$. That $D$ belongs to $\dwa(\sigma_0)$ is obvious. Let now $\alpha,\beta>0$ and observe that, for all $x\geq 1$, 
\begin{align*}
\sum_{\log\log(n+\gamma)\in\left[x;x+\frac{\alpha}{x^2}\right]}\frac 1n&\geq C(\gamma)\left(\exp\left(x+\frac\alpha{x^2}\right)-\exp(x)\right)\\
&\geq C(\gamma,\alpha) \frac{\exp(x)}{x^2}\\
&\geq C(\gamma,\alpha,\beta) \exp\big((1-\beta)x\big)
\end{align*}
so that $D\in \ddens(\sigma_0)$.
\end{proof}

\subsection{The prime number series}\label{sec:primeuniversality}

We start with some general considerations. Let $\lambda$ be a frequency, let $D(s)= \sum_n a(n)e^{-\lambda(n)s}$ be a Dirichlet series,
let $\sigma_0\in\mathbb R$ and let $\Omega$ be the half-plane $\Omega=\{\Re e(s)>\sigma_0\}$. If we assume that $\sum_n |a(n)|^2 e^{-2\lambda(n)\sigma_0}<+\infty$
for all $\sigma>\sigma_0$, then as detailed in Section \ref{SEC:CVMEASURE}, $D(s,z)=\sum_n a(n)z_n e^{-\lambda(n)s}$ defines an $H(\Omega)$-valued
random element to which we can associate its distribution $P_D$. If we assume that $\sigma_a(D)\leq\sigma_0$, then clearly $D\in\dwa(\sigma_0)$ and the measure $\nu_{T,D}$
defined by \eqref{eq:nutd} converges weakly to $P_D$ (see also \cite[Theorem 5.1]{Duy12}).
Suppose now that we have a Dirichlet series $D(s)=\sum_n a(n)e^{-\lambda(n)s}$ and a sequence of Dirichlet series $D_N(s)=\sum_n a_N(n)e^{-\lambda(n)s}$
with $\sigma_2(D),\ \sigma_2(D_N)\leq\sigma_0$ for all $N$. If $(D_N)$ converges to $D$ in the sense that, for all $\sigma>\sigma_0$, 
$\sum_n |a(n)-a_N(n)|^2 e^{-2\lambda(n)\sigma}$ tends to $0$ as $N$ tends to $+\infty$, then a careful look at the proof of \cite[Theorem 5.4]{Duy12}
shows that $(P_{D_N})$ converges weakly to $P_D$. For instance this holds true provided $a_N(n)$ tends to $a(n)$ with $|a_N(n)|\leq |a(n)|$ for all $n,N$.

We now restrict our discussion to the prime number series $D(s)=\sum_{k\geq 1}p_k^{-s}$ and let $(D_N)$ be its sequence of partial sums, 
$D_N(s)=\sum_{k=1}^N p_k^{-s}$. The Zeta-function and $D$ are related by the formula
$$\log\zeta(s)=\sum_{j=1}^{+\infty}\frac{D(js)}{j},\ \Re e(s)>1.$$
By Möbius inversion formula, 
\begin{equation}\label{eq:mobiusinversion}
D(s)=\sum_{k=1}^{+\infty}\mu(k)\frac{\log \zeta(ks)}{k},\ \Re e(s)>1.
\end{equation}
In a similar way, if $\zeta_N(s)=\prod_{k=1}^N \frac{1}{1-p_k^{-s}}$, then 
$$D_N(s)=\sum_{k=1}^{+\infty}\mu(k)\frac{\log\zeta_N(ks)}{k},\ \Re e(s)>1.$$
Relation \eqref{eq:mobiusinversion} is the easiest way to define $D$ on $\mathbb C_{1/2}$. Indeed, $\log\zeta$ can be defined holomorphically on $\mathbb C_{1/2}$ if we remove all horizontal line segments from the lines $\Re (s)=1/2$ and the poles and zeros of $\zeta$ (if any!) in this half-plane. The remaining part, $\sum_{k\geq 2}\mu(k) k^{-1}\log\zeta(ks)$ is an absolutely convergent Dirichlet series in $\mathbb C_{1/2}$ (see below).

One has to be careful if we want to define the measure $\nu_{T,D}$ as before. Indeed, even if we fix some open subset $U$ of the strip $\Omega=\{1/2<\Re e(s)<1\}$ with $\overline U$ a compact subset of $\Omega$, then $D(\cdot+i\tau)$ may not belong to $H(U)$ for an infinite number of values of $\tau$. However, the Bohr-Landau theorem asserts that $N(\sigma,T)=o(T)$ for all $\sigma>1/2$ where $N(\sigma,T)$ denotes the number of zeros of $\zeta$ with $\Re e(s)>\sigma$ and $|\Im m(s)|\leq T$. Therefore, $D(\cdot+i\tau)$ is holomorphic in $U$ for most of the values of $\tau$.

To be more precised, we now fix $U$, $V$ two open rectangles in $\Omega$ with $\bar U\subset V\subset \bar V\subset \Omega$. For $T>0$, let $\mathcal N_{D,V}=\{\tau>0:\ D(\cdot+i\tau)\in H(V)\}$ which has density $1$. We set, for $A\in\mathcal B(H(U))$, 
$$\tilde \nu_{T,D,U}(A)=\frac{1}{\textrm{meas}(\mathcal N_{D,V}\cap[0,T])}\textrm{meas}\left(\tau\in[0,T]\cap \mathcal N_{D,V}:\ D(\cdot+i\tau)\in A\right).$$
Taking into account that $P_{D,U}$ has full support in $H(U)$ (this follows from an argument similar to Example \ref{ex:primerearrangement}), Theorem \ref{thm:primeuniversal} will be a consequence of the following proposition:

\begin{proposition}\label{prop:primemeasure}
The sequence of measures $(\tilde \nu_{T,D,U})_{T>0}$ converges weakly to $P_{D,U}$. 
\end{proposition}

The proof of Proposition \ref{prop:primemeasure} is based on an integral estimate. For $T>0$, we denote $A_T=\mathcal N_{D,V}\cap [0,T]$.
\begin{lemma}\label{lem:primemeasure}
$$\lim_{N\to+\infty}\limsup_{T\to+\infty}\frac{1}{\mathrm{meas}(A_T)}\int_{A_T}\int_V |D(s+i\tau)-D_N(s+i\tau)|^2 dsd\tau=0.$$
\end{lemma}
\begin{proof}
We write 
\begin{align*}
|D(s+i\tau)-D_N(s+i\tau)|&\leq |\log\zeta(s+i\tau)-\log\zeta_N(s+i\tau)|+\\
&\quad\quad\sum_{k=2}^{+\infty}k^{-1}|\log \zeta(k(s+i\tau))-\log\zeta_N(k(s+i\tau))|
\end{align*}
which leads to the study of two integrals. That 
$$\lim_{N\to+\infty}\limsup_{T\to+\infty} \frac{1}{\textrm{meas}(A_T)}\int_{A_T}\int_V |\log\zeta(s+i\tau)-\log\zeta_N(s+i\tau)|^2dsd\tau=0$$
can be found in \cite[Theorem 4.1]{Seip20}. Regarding the second integral, we write, setting $\sigma_1=\inf(\Re e(s):\ s\in V)>1/2$, for all $s\in V$,
\begin{align*}
\sum_{k\geq 2} k^{-1}|\log \zeta(k(s+i\tau))-\log\zeta_N(k(s+i\tau))|&\leq \sum_{l\geq N+1}\sum_{j\geq 1}\sum_{k\geq 2}\frac{p_l^{-kj\sigma_1}}{kj}\\
&\leq \sum_{l\geq N+1}\sum_{j\geq 1}\sum_{k\geq 2}p_l^{-kj\sigma_1}\\
&\leq C(\sigma_1) \sum_{l\geq N+1}p_l^{-2\sigma_1}\\
&\leq C(\sigma_1)\veps_N
\end{align*}
where $(\veps_N)$ goes to $0$. Therefore
$$\lim_{N\to+\infty}\limsup_{T\to+\infty}\frac{1}{\textrm{meas}(A_T)}\int_{A_T}\!\!\int_V\left|
\sum_{k\geq 2} \frac{\log \zeta(k(s+i\tau))-\log\zeta_N(k(s+i\tau))}k\right|^2dsd\tau=0.$$
\end{proof}

\begin{proof}[Proof of Proposition \ref{prop:primemeasure}]
The deduction of Proposition \ref{prop:primemeasure} from Lemma \ref{lem:primemeasure} is now routine and we shall be brief. Since convergence in Bergman spaces entails uniform convergence on compact subsets, we know that
$$\lim_{N\to+\infty}\limsup_{T\to+\infty}\frac{1}{\mathrm{meas}(A_T)}\int_{A_T}\sup_{s\in U} |D(s+i\tau)-D_N(s+i\tau)|^2 dsd\tau=0.$$
From Chebychev's inequality, we get that, for all $\veps>0$, 
$$\lim_{N\to+\infty}\limsup_{T\to+\infty}\frac{1}{\mathrm{meas}(A_T)} \mathrm{meas}\left(\tau\in A_T:\ \sup_{s\in U}|D(s+i\tau)-D_N(s+i\tau)|>\veps\right)=0.$$
Now we have already shown that $P_{D_N,U}\to P_{D,U}$ weakly as $N\to+\infty$ and that, for all $N\in\mathbb N$, $\nu_{T,D_N,U}\to P_{D_N,U}$ weakly as $T\to+\infty$. The proposition follows from standard facts about
weak convergence of probability measures (see \cite[Theorem 3.2]{Bil68} and \cite[Theorem 5.5]{Duy12}) and from the fact that $\mathcal N_{D,V}$ has density $1$.
\end{proof}

Our proof essentially says that a perturbation of $\log\zeta$ by a Dirichlet series which is absolutely convergent in $\mathbb C_{1/2}$ is still strongly universal in the critical strip. However, we had to reproduce and modify one proof of strong universality of $\log\zeta$. This can be done for other universal Dirichlet series and this motivates the following question:
\begin{question}
Let $D$ be a Dirichlet series which is (strongly) universal in some strip and let $D_0$ be a Dirichlet series (with the same frequencies) which is absolutely convergent in this strip. Is $D+D_0$ still (strongly) universal?
\end{question}

Our method seems to break down for the alternate prime series.
\begin{question}
Let $D(s)=\sum_{n\geq 1}(-1)^n p_n^{-s}$. Is $D$ strongly universal in $\{1/2<\Re e(s)<1\}$?
\end{question}


\begin{thebibliography}{10}

\bibitem{Ban90}
W.~Banaszczyk, \emph{{The Steinitz theorem on rearrangement of series for
  nuclear spaces}}, J. Reine. Angew. Math. (1990), 187--200.

\bibitem{BM09}
F.~Bayart and \'E. Matheron, \emph{Dynamics of linear operators}, Cambridge
  Tracts in Math, vol. 179, Cambridge University Press, 2009.

\bibitem{Bil68}
P.~Bilingsley, \emph{{Convergence of probability measures}},  (1968).

\bibitem{Duy12}
T.K. {Duy}, \emph{{Remarks on value distributions of general Dirichlet
  series}}, {RIMS K\^oky\^uroku Bessatsu} \textbf{B34} (2012), 49--68.

\bibitem{GLS06}
J.~Genys, A.~Laurin{\v{c}}ikas, and D.~{\v S}iau{\v{c}}i{\={u}}nas,
  \emph{{Value distribution of general Dirichlet series. VII}}, Lith. Math. J.
  \textbf{46}, 155--162.

\bibitem{GM21}
G.~Giorgobiani and N.~Manjavidze, \emph{{Rearrangement universality of the
  Dirichlet and prime numbers' powers series}}, preprint.

\bibitem{HarRi}
G.~H. Hardy and M.~Riesz, \emph{{The general theory of Dirichlet series}},
  vol.~18, Cambridge University Press, 1915.

\bibitem{IK04}
H.~Iwaniec and E.~Kowalski, \emph{Analytic number theory}, American
  Mathematical Society Colloquium Publications, American Mathematical Society,
  2004.

\bibitem{Jarbook}
H.~Jarchow, \emph{{Locally convex spaces}}, Teubner, 1981.

\bibitem{Lau96}
A.~Laurin{$\check{\rm c}$}ikas, \emph{Limit {Theorems for the Riemann zeta
  function}}, Mathematics and its Applications, vol. 352, Kluwer Academic
  Press, 1996.

\bibitem{LSS03}
A.~Laurin{$\check{\rm c}$}ikas, W.~Schwarz, and J.~Steuding, \emph{{The
  universality of general Dirichlet series}}, Analysis \textbf{23} (2003),
  13--26.

\bibitem{GaLa02}
A~Laurin{\v{c}}ikas and R.~Garunk{\v{s}}tis, \emph{{The Lerch Zeta-function}},
  Kluwer, 2002.

\bibitem{Matsu15}
K.~Matsumoto, \emph{{A survey on the theory of universality for zeta and
  L-functions}}, "Number Theory: Plowing and Starring through High Wave Forms",
  Proc. 7th China-Japan Seminar, World Scientific, 2015, pp.~95--144.

\bibitem{MV74}
H.L. Montgomery and R.C. Vaughan, \emph{{Hilbert's inequality}}, J. Lond. Math.
  Soc. (1974), 73--82.

\bibitem{Seip20}
K.~Seip, \emph{Universality and distribution of zeros and poles of some zeta
  functions}, J. d'Anal. Math. \textbf{141} (2020), 331--381.

\bibitem{VORONIN}
S.~M. Voronin, \emph{A theorem on the ``universality'' of the {R}iemann zeta
  function}, Math. USSR-Izv. \textbf{9} (1975), 443--453.

\end{thebibliography}

\providecommand{\bysame}{\leavevmode\hbox to3em{\hrulefill}\thinspace}
\providecommand{\MR}{\relax\ifhmode\unskip\space\fi MR }
\providecommand{\MRhref}[2]{%
  \href{http://www.ams.org/mathscinet-getitem?mr=#1}{#2}
}
\providecommand{\href}[2]{#2}

\end{document}